\documentclass[11pt,a4paper]{amsart}

\usepackage{amsmath, amsthm, amsfonts, amssymb,color,geometry,enumerate}
\usepackage{graphicx}
\usepackage{fancybox}
\usepackage{cases}
\usepackage{fancyhdr}

\theoremstyle{definition}

\newtheorem{theorem}{Theorem}[section]
\newtheorem{proposition}[theorem]{Proposition}
\newtheorem{lemma}[theorem]{Lemma}
\newtheorem{definition}[theorem]{Definition}

\newtheorem{problem}[theorem]{Problem}

\makeatletter
\renewcommand{\section}{%
  \@startsection{section}%
   {1}%
   {\z@}%
   {-3.5ex \@plus -1ex \@minus -.2ex}%
   {2.3ex \@plus.2ex}%
   {\sc\normalsize \center }%
}%
\makeatother

\makeatletter
\renewcommand{\subsection}{%
  \@startsection{subsection}%
   {1}%
   {\z@}%
   {-3.5ex \@plus -1ex \@minus -.2ex}%
   {2.3ex \@plus.2ex}%
   {\sc\normalsize \center}%
}%
\makeatother

\makeatletter
\renewcommand{\subsubsection}{%
  \@startsection{subsubsection}%
   {1}%
   {\z@}%
   {-3.5ex \@plus -1ex \@minus -.2ex}%
   {2.3ex \@plus.2ex}%
   {\sc\normalsize \center}%
}%
\makeatother

\makeatletter

\@addtoreset{equation}{section}
\makeatother

\renewcommand{\epsilon}{\varepsilon}    % epsilon 
\begin{document}

\title{{\bf \sc On tensors of factorizable quantum channels with the completely depolarizing channel}}
\author{Yuki Ueda}
\date{}

\address{
Yuki Ueda\footnote{This research was completed during the author's affiliation with Kyoto University and Hokkaido University from 2016 to 2018.}:
Department of Mathematics, Hokkaido University of Education, 9, Hokumon-cho, Asahikawa, Hokkaido, 070-8621, Japan}
\email{ueda.yuki@a.hokkyodai.ac.jp (current)}
%\email{ueda.yuki@a.hokkyodai.ac.jp}

\maketitle

\begin{abstract}
In this paper, we obtain results for factorizability of quantum channels. Firstly, we prove that if a tensor $T\otimes S_k$ of a quantum channel $T$ on $M_n(\mathbb{C})$ with the completely depolarizing channel $S_k$ is written as a convex combination of automorphisms on the matrix algebra $M_n(\mathbb{C})\otimes M_k(\mathbb{C})$ with rational coefficients, then the quantum channel $T$ has an exact factorization through some matrix algebra with the normalized trace. Next, we prove that if a quantum channel has an exact factorization through a finite dimensional von Neumann algebra with a convex combination of normal faithful tracial states with rational coefficients, then it also has an exact factorization through some matrix algebra with the normalized trace.
\end{abstract}

{\it keywords: Markov maps, factorizable quantum channels, free group von Neumann algebras, completely depolarizing channels}

\section*{Note}
This version does not differ significantly from the one initially posted in 2018. However, we would like to inform readers that Problems 1.1 and 1.3 below have already been completely solved in the published version.

\section{Introduction}
In \cite{AD}, Anantharaman-Delaroche defined the class of factorizable Markov maps to study the noncommutative analogue of ergodic theory. After the works of \cite{AD}, Haagerup and Musat proved in \cite[Theorem 6.1]{HM11} that every non-factorizable quantum channel on $M_n(\mathbb{C})$ ($n\ge 3$) fails the asymptotic quantum Birkhoff conjecture which was raised by Smolin, Verstraete and Winter (see \cite{SVW}) as one of the most important problems in quantum information theory. In \cite{HM11, HM15}, they also approached the Connes embedding problem by using factorizable quantum channels, in particular, tensors of factorizable quantum channels with the completely depolarizing channel. In this paper, we focus on the relation between the factorizability of quantum channels and the property of tensors of factorizable quantum channels with the completely depolarizing channel. Haagerup and Musat proved in \cite[Corollary 3.5]{HM15} that if a quantum channel $T$ on $M_n(\mathbb{C})$ has an exact factorization through a tracial $W^\ast$-probability space $(M_n(\mathbb{C})\otimes M_k(\mathbb{C}),\tau_n\otimes \tau_k)$, that is, there exists a unitary matrix $u$ in $M_n(\mathbb{C})\otimes M_k(\mathbb{C})$ such that  $T$ is written in the following form
\begin{equation}\label{EF}
Tx=(id_n\otimes \tau_k)(u^\ast (x\otimes 1_k)u), \qquad x\in M_n(\mathbb{C}),
\end{equation}
then $T\otimes S_k \in$ conv(Aut($M_n(\mathbb{C})\otimes M_k(\mathbb{C})$)). We raise the natural problem of whether the converse claim of the statement is true or not (see below).
\begin{problem}\label{prob}
Let $n$ be a positive integer and $T$ a quantum channel on $M_n(\mathbb{C})$. Is it true that the following properties are equivalent?
\begin{enumerate}[\rm (1)]
\item There exists a positive integer $k$ such that $T$ has an exact factorization through $(M_n(\mathbb{C})\otimes M_k(\mathbb{C}),\tau_n\otimes \tau_k)$.
\item There exists a positive integer $k$ such that $T\otimes S_k\in$conv(Aut($M_n(\mathbb{C})\otimes M_k(\mathbb{C})$)).
\end{enumerate}
The implication $(1)\Rightarrow (2)$ is true, but the implication $(2)\Rightarrow (1)$ is unknown.
\end{problem}
We obtain a partial result for Problem \ref{prob} as follows.
\begin{theorem}\label{TM11}
Let $T$ be a quantum channel on $M_n(\mathbb{C})$. If there exists a positive integer $k$ such that $T\otimes S_k=\sum_{i=1}^{d(k)} c_i \text{ad}(u_i)\in$conv(Aut($M_n(\mathbb{C})\otimes M_k(\mathbb{C})$)) for some positive integer $d(k)$, unitary matrices $u_1,\cdots ,u_{d(k)}\in\mathcal{U}(M_n(\mathbb{C})\otimes M_k(\mathbb{C}))$ and positive rational numbers $c_1,\cdots ,c_{d(k)}$ with $\sum_{i=1}^{d(k)}c_i=1$, then there exists a positive integer $L$ such that $T$ has an exact factorization through $(M_n(\mathbb{C})\otimes M_L(\mathbb{C}),\tau_n\otimes \tau_L)$.
\end{theorem} 
Moreover we also raise the following problem for the quantum channels which have an exact factorization through a finite dimensional $W^\ast$-probability space (see below).
\begin{problem}
Let $T$ be a quantum channel on $M_n(\mathbb{C})$. Is it true that if there exists a finite dimensional $W^\ast$-probability space $(\mathcal{N},\phi)$ (i.e. a pair of a finite dimensional von Neumann algebra $\mathcal{N}$ and a normal faithful state $\phi$ on $\mathcal{N}$) such that $T$ has an exact factorization through $(M_n(\mathbb{C})\otimes \mathcal{N},\tau_n\otimes \phi)$ then there exists a positive integer $k$ such that $T$ also has an exact factorization through $(M_n(\mathbb{C})\otimes M_k(\mathbb{C}),\tau_n\otimes \tau_k)$?
\end{problem}
Note that every finite dimensional von Neumann algebra is $\ast$-isomorphic to a direct sum of some matrix algebras. For this problem, we obtain the following partial result.
\begin{theorem}\label{TM13}
Let $T$ be a quantum channel on $M_n(\mathbb{C})$. If for any $\alpha:=(\alpha_1,\cdots ,\alpha_d) \in \mathbb{Q}_+^d$ with $\alpha_1+\cdots +\alpha_d=1$ there exist positive integers $k_1,\cdots ,k_d$ such that $T$ has an exact factorization through $(M_n(\mathbb{C}) \otimes (M_{k_1}(\mathbb{C})\oplus \cdots \oplus M_{k_d}(\mathbb{C})),\tau_n\otimes\tau_\alpha)$, where $\tau_\alpha$ is a normal faithful tracial state on $ M_{k_1}(\mathbb{C})\oplus \cdots \oplus M_{k_d}(\mathbb{C})$ defined by 
\begin{equation}\label{tau}
\tau_{\alpha}(x_1,\cdots ,x_d):=\alpha_1 \tau_{k_1}(x_1)+ \cdots + \alpha_d \tau_{k_d}(x_d)
\end{equation}
for all $(x_1,\cdots ,x_d)\in M_{k_1}(\mathbb{C})\oplus \cdots \oplus M_{k_d}(\mathbb{C})$, then there exists a positive integer $k$ such that $T$ has an exact factorization through $(M_n(\mathbb{C})\otimes M_k(\mathbb{C}),\tau_n\otimes \tau_k)$.
\end{theorem}
In section 2, we set notations and definitions in this paper. In section 3, we recall and discuss for the concepts and basic properties of quantum channels, factorizable quantum channels and completely depolarizing channels. In section 4 and 5, we prove that Theorem \ref{TM11} and Theorem \ref{TM13} hold, respectively.

\section{Notations and Definitions}
In this paper, we use the following notations:
\begin{itemize}
\item $\mathbb{N}:=\{1,2,3,\cdots\}$ and $\mathbb{Q}_+$ is the set of all positive rational numbers.
\item $M_n(\mathbb{C})$ is the set of all $n\times n$ matrices with complex entries.
\item $\mathcal{U}(n)$ is the set of all $n\times n$ unitary matrices with complex entries.
\item $\tau_n$ is the normalized trace on $M_n(\mathbb{C})$, i.e. $\tau_n((x_{ij})_{1\le i,j\le n}):=\frac{x_{11}+\cdots x_{nn}}{n}$.
\item $id_n$ is the identity map on $M_n(\mathbb{C})$.
\item Define ad$(u)(x):=u^\ast x u$ for all $x\in M_n(\mathbb{C})$, $u\in \mathcal{U}(n)$.
\item conv(Aut($M_n(\mathbb{C})$)) is the convex hull of the set Aut$(M_n(\mathbb{C})):=\{\text{ad}(u): u\in\mathcal{U}(n)\}$.
\item $1_\mathcal{M}$ is the unit in von Neumann algebra $\mathcal{M}$, in particular, $1_n:=1_{M_n(\mathbb{C})}$.
\end{itemize}
A pair $(\mathcal{M},\phi)$ is called a {\it $W^\ast$-probability space} if $\mathcal{M}$ is a von Neumann algebra and $\phi$ is a normal faithful state on $\mathcal{M}$. In particular, we call $(\mathcal{M},\phi)$ a {\it tracial} $W^\ast$-probability space when $\phi$ is tracial, that is $\phi(xy)=\phi(yx)$ for all $x,y\in \mathcal{M}$. 

\section{Basic properties of factorizable quantum channels}
In \cite{AD}, Anantharaman-Delaroche considered factorizable Markov maps to prove a noncommutative analogue of Rota's theorem. We first recall the definition of Markov maps on a $W^\ast$-probability space. The concept is a noncommutative analogue of measure-preserving Markov operator on a probability space. 
\begin{definition}\label{Markov}
Let $(\mathcal{M},\phi)$ and $(\mathcal{N},\psi)$ be $W^\ast$-probability spaces. A linear map $T:\mathcal{M}\rightarrow\mathcal{N}$ is called a {\it $(\phi,\psi)$-Markov map} if
\begin{enumerate}[\rm (1)]
\item $T$ is completely positive,
\item $T$ is unital,
\item $T$ is $(\phi,\psi)$-preserving, i.e. $\psi\circ T=\phi$,
\item $T\circ \sigma_t^{\phi}=\sigma_t^{\psi}\circ T$, where $\{\sigma_t^{\phi}\}_{t\in\mathbb{R}}$ denotes the automorphism group of the state $\phi$.
\end{enumerate}
In particular, we call it {\it $\phi$-Markov map} when $(\mathcal{M},\phi)=(\mathcal{N},\psi)$.
\end{definition}
If $(\mathcal{M},\phi)=(\mathcal{N},\psi)=(M_n(\mathbb{C}),\tau_n)$ in Definition \ref{Markov}, the fourth condition is removed since the operator $\sigma_t^{\tau}$ is trivial for any normal faithful tracial states $\tau$ on von Neumann algebras and $t\in\mathbb{R}$, so that a $\tau_n$-Markov map means a unital completely positive trace-preserving map ({\it quantum channel}) on $M_n(\mathbb{C})$. Denote by
\begin{equation}
\mathcal{Q}(n):=\{T:M_n(\mathbb{C})\rightarrow M_n(\mathbb{C}): T \text{ is a quantum channel}\}.
\end{equation}
In \cite[Definition 6.2]{AD}, Anantharaman-Delaroche defined the class of factorizable Markov maps in the following sense.
\begin{definition}\label{factorizable}
A $(\phi,\psi)$-Markov map $T:\mathcal{M}\rightarrow \mathcal{N}$ is called {\it factorizable} if there exists a $W^\ast$-probability space $(\mathcal{L},\chi)$ and $\ast$-monomorphisms $\alpha:\mathcal{M}\rightarrow \mathcal{L}$ and $\beta:\mathcal{N}\rightarrow\mathcal{L}$ such that $\alpha$ is $(\phi,\chi)$-Markov, $\beta$ is $(\psi,\chi)$-Markov and $T=\beta^\ast\circ \alpha$, where $\beta^\ast:\mathcal{L}\rightarrow\mathcal{M}$ is the adjoint of $\beta$ (see \cite[Remark 1.2]{HM11}). 
\end{definition}
The set of all factorizable $(\phi,\psi)$-Markov maps is closed under composition, the adjoint operation, taking convex combinations and $w^\ast$-limits (See \cite[Proposition 2]{R08}).  Haagerup and Musat proved in \cite[Theorem 2.2]{HM11} the following statement for the class of factorizable quantum channels.
\begin{proposition}\label{HM}
Consider $T\in \mathcal{Q}(n)$. Then the following properties are equivalent:
\begin{enumerate}[\rm (1)]
\item $T$ is factorizable,
\item There exists a tracial $W^\ast$-probability space $(\mathcal{M},\phi)$ and a unitary $u$ in $M_n(\mathbb{C})\otimes \mathcal{M}$ such that
\begin{equation}
Tx=(id_n\otimes \phi)(u^\ast (x\otimes 1_\mathcal{M})u), \qquad x\in M_n(\mathbb{C}).
\end{equation}
\end{enumerate}
\end{proposition} 
In this case, we say that {\it $T$ has an exact factorization through $(M_n(\mathbb{C})\otimes \mathcal{M}, \tau_n\otimes \phi)$}. A factorization of quantum channels is not unique. We have two examples of a factorization of quantum channels. Firstly we show the following statement.
\begin{lemma}\label{Preserv}
If $T\in\mathcal{Q}(n)$ has an exact factorization through a tracial $W^\ast$-probability space $(M_n(\mathbb{C}) \otimes \mathcal{M},\tau_n\otimes \phi)$ and there exist a tracial $W^\ast$-probability space $(\mathcal{N},\psi)$ and a $(\phi,\psi)$-Markov *-homomorphism $S: (\mathcal{M},\phi) \rightarrow (\mathcal{N},\psi)$, then $T$ also has an exact factorization through $(M_n(\mathbb{C}) \otimes \mathcal{N},\tau_n\otimes \psi)$.
\end{lemma}
\begin{proof}
Since $T$ has an exact factorization through $(M_n(\mathbb{C}) \otimes \mathcal{M},\tau_n\otimes\phi)$, there exists a unitary $u \in M_n(\mathbb{C}) \otimes \mathcal{M}$ such that
\begin{equation}
Tx=(id_n \otimes \phi)(u^*(x \otimes 1_{\mathcal{M}})u), \qquad x \in M_n(\mathbb{C}).
\end{equation}
Since $S$ is a *-homomorphism, $(id_n \otimes S)(u)$ is a unitary in $M_n(\mathbb{C}) \otimes \mathcal{N}$. Hence,
\begin{equation}
\begin{split}
(id_n \otimes \psi)\bigl((id_n \otimes S)(u)^*(x \otimes 1_{\mathcal{N}})(id_n \otimes S)(u)\bigr) &= (id_n \otimes \psi)(id_n \otimes S)(u^*(x \otimes 1_{\mathcal{M}})u)\\
&=(id_n \otimes \phi)(u^*(x \otimes 1_{\mathcal{M}})u)=Tx,
\end{split}
\end{equation}
for all $x \in M_n(\mathbb{C}).$ Therefore $T$ has an exact factorization through $(M_n(\mathbb{C}) \otimes \mathcal{N},\tau_n\otimes \psi)$.
\end{proof}
As the second example, we consider a linear map $T$ defined by
\begin{equation}
Tx=\sum_{i=1}^2 a_i^\ast x a_i, \qquad x\in M_2(\mathbb{C}),
\end{equation}
where
\begin{equation}
a_1:=
\begin{pmatrix}
1 & 0\\
0 & 0
\end{pmatrix}, \qquad
a_2:=
\begin{pmatrix}
0 & 0\\
0 & 1
\end{pmatrix},
\end{equation}
has an exact factorization through $M_2(\mathbb{C})\otimes \mathcal{L}\mathbb{F}_2$. It is clear that $\sum_{i=1}^2a_i^\ast a_i=\sum_{i=1}^2a_ia_i^\ast=1_2$. Hence $T$ is a quantum channel on $M_2(\mathbb{C})$ by \cite{Choi}. Let $g_1,g_2$ be generators of the free group $\mathbb{F}_2$ of degree $2$ and set $u:=\sum_{i=1}^2 a_i\otimes \lambda_{g_i} \in \mathcal{U}(M_2(\mathbb{C})\otimes \mathcal{L}\mathbb{F}_2$), where $\lambda_g$ is the left representation of $g\in\mathbb{F}_2$, that is, 
\begin{equation}
\lambda_g(f)(h):=f(g^{-1}h), \qquad f\in \l^2\mathbb{F}_2, \hspace{2mm} g,h\in\mathbb{F}_2,
\end{equation}
and $\mathcal{L}\mathbb{F}_2$ is the free group von Neumann algebra. For all $x\in M_2(\mathbb{C})$,
\begin{equation}
(id_2 \otimes \tau_{\mathcal{L}\mathbb{F}_2})(u^\ast (x\otimes 1_{\mathcal{L}\mathbb{F}_2})u)=\sum_{i,j=1}^2 \tau_{\mathcal{L}\mathbb{F}_2}(\lambda_{g_i}^\ast \lambda_{g_j})a_i^\ast xa_j=\sum_{i,j=1}^2 \delta_{ij}a_i^\ast x a_j=Tx,
\end{equation}
where 
\begin{equation}
\tau_{\mathcal{L}\mathbb{F}_2}(\lambda):=\langle\lambda \delta_e, \delta_e\rangle_{l^2\mathbb{F}_2}, \qquad \lambda \in \mathcal{L}\mathbb{F}_2,
\end{equation}
and $e\in\mathbb{F}_2$ is the unit of $\mathbb{F}_2$. Therefore $T$ has an exact factorization through $(M_2(\mathbb{C})\otimes \mathcal{L}\mathbb{F}_2, \tau_2 \otimes \tau_{\mathcal{L}\mathbb{F}_2})$. On the other hand, $T$ also has an exact factorization through $(M_2(\mathbb{C})\otimes M_4(\mathbb{C}), \tau_2\otimes \tau_4)$ since $a_1,a_2\in M_2(\mathbb{C})$ are self-adjoint, $a_1^2+a_2^2=1_2$ and $a_1a_2=a_2a_1$ and therefore we can apply \cite[Corollary 2.5]{HM11} to the quantum channel $T$. Similarly we have the following statement.
\begin{proposition}
Consider $d\ge 2$. Let $T$ be a quantum channel on $M_d(\mathbb{C})$ defined by
\begin{equation}
Tx:=\sum_{i=1}^d E_{ii}^\ast x E_{ii}, \qquad x\in M_d(\mathbb{C}),
\end{equation}
where $\{E_{ij}\}_{1\le i,j \le d}$ is the set of standard matrix units in $M_d(\mathbb{C})$. Then the following conditions hold.
\begin{enumerate}[\rm (1)]
\item $T$ has an exact factorization through $(M_d(\mathbb{C})\otimes \mathcal{L}\mathbb{F}_d,\tau_d\otimes \tau_{\mathcal{L}\mathbb{F}_d})$.
\item $T$ has an exact factorization through $(M_d(\mathbb{C})\otimes M_{2^d}(\mathbb{C}),\tau_d\otimes\tau_{2^d})$.
\end{enumerate}
\end{proposition}
Denote by
\begin{equation}
\mathcal{F}(n):=\{T\in \mathcal{Q}(n): T \text{ is factorizable}\}.
\end{equation}
By the statements before Proposition \ref{HM}, we have that  conv(Aut($M_n(\mathbb{C})$))$\subset \mathcal{F}(n)$ for all positive integers $n$. Haagerup and Musat found a quantum channel in $\mathcal{F}(n)\setminus$conv(Aut($M_n(\mathbb{C})$)) in \cite[Example 3.3]{HM11}. In particular, K\"{u}mmerer proved in \cite{Kum} that conv(Aut($M_2(\mathbb{C})$))=$\mathcal{F}(2)$. Haagerup and Musat pointed out the important relations between the factorizable quantum channels and the Connes embedding problem in \cite{HM11,HM15}. \\\\
Recall the completely depolarizing channels. Let $S_k:M_n(\mathbb{C})\rightarrow M_n(\mathbb{C})$ be a linear map defined by
\begin{equation}
S_k(x):=\tau_k(x) 1_k, \qquad x\in M_k(\mathbb{C}).
\end{equation}
The map $S_k$ is called the {\it completely depolarizing channel} on $M_n(\mathbb{C})$. Note that $S_k$ is in\\conv(Aut($M_k(\mathbb{C})$)) and therefore it is a factorizable quantum channel on $M_n(\mathbb{C})$. It is clear that $S_k\otimes S_l=S_{kl}$ for all $k,l\in\mathbb{N}$. It is very important to understand tensors $T\otimes S_k$ of a quantum channel $T$ with the completely depolarizing channel $S_k$ to know what $T$ has an exact factorization through some $W^\ast$-probability space.  By \cite[Corollary 2.5]{HM11}, Haagerup and Musat found a quantum channel $T\in\mathcal{F}(n)\setminus$ conv(Aut($M_n(\mathbb{C})$)) $(n\ge3)$ such that it has an exact factorization through $(M_n(\mathbb{C})\otimes  M_{2^d}(\mathbb{C}), \tau_n\otimes\tau_{2^d})$ for some $d\ge 3$. By \cite[Corollary 3.5]{HM15}, we have that $T\otimes S_{2^d} \in$ conv(Aut($M_n(\mathbb{C})\otimes M_{2^d}(\mathbb{C})$)).

\section{Proof of Theorem 1.2}
We prove that Theorem \ref{TM11} holds. We first introduce the following sets. 
\begin{equation}
\begin{split}
 \mathcal{I}_n&:= \left\{ T \in \mathcal{F}(n) \left|
 \begin{array}{l}
  \exists k \in \mathbb{N} \text{ s.t. }  T\otimes S_k=\sum_{i=1}^{d(k)}c_i\text{ad}(u_i)\\
\text{for some positive integer } d(k),\\
  c_1,\cdots, c_{d(k)} \in \mathbb{Q}_+ \text{ with }\sum_{i=1}^{d(k)}c_i=1,\\
  \text{and } u_1,\cdots, u_{d(k)} \in \mathcal{U}(M_n(\mathbb{C}) \otimes M_k(\mathbb{C}))
 \end{array}
 \right.\right\},\\
\mathcal{J}_n&:= \left\{ T \in \mathcal{F}(n) \left|
 \begin{array}{l}
  \exists k \in \mathbb{N} \text{ s.t. }  T \text{ has an exact factorization through}\\
(M_n(\mathbb{C})\otimes M_k(\mathbb{C}),\tau_n\otimes\tau_k)
 \end{array}
 \right.\right\} .
\end{split}
 \end{equation}
By using these notations, we can rewrite as the statement of Theorem \ref{TM11}.
\begin{theorem}\label{AnQ}
Let $n \in \mathbb{N}$. Then $\mathcal{I}_n \subset \mathcal{J}_n$ holds.
\end{theorem}
\begin{proof}
Suppose that $n \in \mathbb{N}$. If $T \in \mathcal{I}_n$, then there is a positive integer $k>0$ such that $T \otimes S_k(z)=\sum_{i=1}^{d(k)}c_iu_i^*zu_i$, for all $z \in M_n(\mathbb{C}) \otimes M_k(\mathbb{C})$ for some positive integer $d(k)>0$, unitaries $u_1, \cdots, u_{d(k)} \in M_n(\mathbb{C}) \otimes M_k(\mathbb{C})$, and positive rational numbers $c_1, \cdots, c_{d(k)}>0$ with $\sum_{i=1}^{d(k)}c_i=1$.  Therefore
\begin{equation}
Tx=(id_n \otimes \tau_k)(T \otimes S_k)(x \otimes 1_k)=\sum_{i=1}^{d(k)}c_i(id_n \otimes \tau_k)(u_i^*(x \otimes 1_k) u_i), \hspace{2mm} x \in M_n(\mathbb{C}).
\end{equation}
Since $c_i$ is a rational number for each $1\le i \le d(k)$, we suppose that
\begin{equation}
c_i=\frac{l_i}{L_i}, \qquad 1\le i \le d(k),
\end{equation}
where $l_i$ and $L_i$ are relatively prime positive numbers for each $i$. Let $L$ be the least common multiple of $L_1,L_2,\cdots ,L_{d(k)}$. Then we can rewrite as:
\begin{equation}
c_i=\frac{l_i\times \frac{L}{L_i}}{L}, \qquad 1\le i \le d(k)
\end{equation}
and rewrite as $C_i:=l_i \times \frac{L}{L_i} \in \mathbb{N}$ for each $i$. Note that $C_1+C_2+\cdots+C_{d(k)}=L$. \\
Then we set 
\begin{equation}
U:=\text{diag}(\underbrace{u_1, \cdots ,u_1}_{C_1}, \underbrace{u_2 \cdots ,u_2}_{C_2}, \cdots, \underbrace{u_{d(k)}, \cdots ,u_{d(k)}}_{C_{d(k)}}) \in M_n(\mathbb{C}) \otimes M_k(\mathbb{C}) \otimes M_L(\mathbb{C}).
\end{equation}
Clearly the block matrix $U$ is a unitary in $M_n(\mathbb{C}) \otimes M_k(\mathbb{C}) \otimes M_L(\mathbb{C})$, and we have that
\begin{equation}
\begin{split}
(id_n& \otimes \tau_k \otimes \tau_L)(U^*(x \otimes 1_k \otimes 1_L)U)\\
&=(id_n \otimes \tau_k \otimes \tau_L) \left( 
\begin{array}{l}
\text{diag}\bigl(\underbrace{u_1^*(x \otimes 1_k)u_1, \cdots ,u_1^*(x \otimes 1_k)u_1}_{C_1}, \cdots\\
\hspace{3cm}\cdots ,\underbrace{u_{d(k)}^*(x \otimes 1_l)u_{d(k)}, \cdots ,u_{d(k)}^*(x \otimes 1_k)u_{d(k)}}_{C_{d(k)}}\bigr)
\end{array} 
\right)\\
&=(id_n \otimes \tau_k \otimes \tau_L) \left(
\begin{array}{l}
\sum_{i=1}^{C_1}u_1^*(x\otimes 1_k)u_1 \otimes E_{ii} + \cdots\\ 
\hspace{2cm}\cdots+\sum_{i=1}^{C_{d(k)}}u_{d(k)}^*(x\otimes 1_k)u_{d(k)} \otimes E_{L-c_{d(k)}+i,L-c_{d(k)}+i}
\end{array} 
\right)\\
&=
\begin{array}{l}
\sum_{i=1}^{C_1}\tau_L(E_{ii})(id_n \otimes \tau_k)(u_1^*(x\otimes 1_k)u_1)+ \cdots\\
\hspace{3cm} \cdots+\sum_{i=1}^{C_{d(k)}} \tau_L(E_{L-c_{d(k)}+i,L-c_{d(k)}+i}) (id_n \otimes \tau_k)(u_{d(k)}^*(x\otimes 1_k)u_{d(k)})
\end{array} \\
&=\sum_{i=1}^{d(k)}\frac{C_i}{L}(id_n \otimes \tau_k)(u_i^*(x\otimes 1_k)u_i)\\
&=\sum_{i=1}^{d(k)}c_i(id_n \otimes \tau_k)(u_i^*(x \otimes 1_k) u_i)=Tx,
\end{split}
\end{equation}
for all $x \in M_n(\mathbb{C})$, where $\{E_{ij}\}_{1 \le i,j \le L}$ is the set of standard matrix units in $M_L(\mathbb{C})$. Therefore $T$ has an exact factorization through $(M_n(\mathbb{C}) \otimes (M_k(\mathbb{C}) \otimes M_L(\mathbb{C})), \tau_n\otimes (\tau_k\otimes\tau_L))$. Hence we have the inclusion $\mathcal{I}_n \subset \mathcal{J}_n$, for each positive integer $n \in \mathbb{N}$. Thus the proof is complete.
\end{proof}

\section{Proof of Theorem 1.4}
We prove that Theorem \ref{TM13} holds. We also denote the following set.
\begin{equation}
\mathcal{K}_n:=\{T\in\mathcal{F}(n): \exists k\in\mathbb{N} \text{ s.t. } T \otimes S_k \in \text{conv(Aut(}M_n(\mathbb{C})\otimes M_k(\mathbb{C})))\}.
\end{equation}
Clearly conv(Aut($M_n(\mathbb{C})$))$\subset \mathcal{K}_n$. The statement before the condition \eqref{EF} in section 1 or \cite[Corollary]{HM15} implies that $\mathcal{J}_n\subset\mathcal{K}_n$ holds for each positive integer $n$.  By the last statements in section 3, the set $\mathcal{K}_n\setminus$conv(Aut($M_n(\mathbb{C})$))) is not empty for all $n\ge 3$. 
We can rewrite as the statement of Theorem \ref{TM13} by the new notations in sections 4 and 5.
\begin{theorem}\label{fin} 
Let $n\in \mathbb{N}$. If for any $\alpha:=(\alpha_1,\cdots ,\alpha_d) \in \mathbb{Q}_+^d$ with $\alpha_1+\cdots +\alpha_d=1$ there exist some positive integers $k_1,\cdots ,k_d$ such that $T\in\mathcal{F}(n)$ has an exact factorization through $(M_n(\mathbb{C}) \otimes (M_{k_1}(\mathbb{C})\oplus \cdots \oplus M_{k_d}(\mathbb{C})),\tau_n\otimes\tau_\alpha)$, where $\tau_\alpha$ is defined by \eqref{tau}, then $T \in \mathcal{J}_n$, and therefore $T\in\mathcal{K}_n$.
\end{theorem}
\begin{proof}
{\bf (Step 1)} Suppose that $\alpha:=(\alpha_1,\cdots ,\alpha_d) \in \mathbb{Q}_+^d$ with $\alpha_1+\cdots +\alpha_d=1$ and  there exist $k\in\mathbb{N}$ and the normal faithful tracial state $\tau_\alpha$ defined by
\begin{equation}
\tau_{\alpha}(x_1,\cdots ,x_d):=\alpha_1 \tau_{k}(x_1) + \cdots + \alpha_d \tau_{k}(x_d), \qquad (x_1,\cdots ,x_d)\in\underbrace{M_k(\mathbb{C}) \oplus \cdots \oplus M_k(\mathbb{C})}_d
\end{equation}
such that $T$ has an exact factorization through $(M_n(\mathbb{C}) \otimes (\underbrace{M_k(\mathbb{C}) \oplus \cdots \oplus M_k(\mathbb{C})}_d),\tau_n\otimes \tau_\alpha)$. Since $\alpha_1,\cdots ,\alpha_d$ are rational numbers, we can write as
\begin{equation}
\alpha_i=\frac{l_i}{L},\qquad 1\le i\le d,
\end{equation}
where $l_i$ and $L$ are positive integers and $l_1+\cdots +l_d=L$. We will define the following map:
\begin{equation}
\phi:\underbrace{M_k(\mathbb{C}) \oplus \cdots \oplus M_k(\mathbb{C})}_d \rightarrow M_{kL}(\mathbb{C}), \qquad (x_1,\cdots, x_d) \mapsto \text{diag}(\underbrace{x_1,\cdots ,x_1}_{l_1}, \cdots ,\underbrace{x_d,\cdots ,x_d}_{l_d})
\end{equation}
By the definition of $\phi$, it is easy to check that $\phi$ is a unital completely positive *-homomorphism. We consider $x_i=(x_{sl}^i)_{1 \le s,l \le k} \in M_k(\mathbb{C})$ $(1 \le i \le d)$. Then 
\begin{equation}
\tau_{kL} \circ \phi (x_1, \cdots ,x_d) = \tau_{kL}(\text{diag}(\underbrace{x_1,\cdots ,x_1}_{l_1}, \cdots ,\underbrace{x_d,\cdots ,x_d}_{l_d})=\frac{1}{kL} \sum_{i=1}^d \sum_{j=1}^k l_ix_{jj}^i.
\end{equation}
On the other hand,
\begin{equation}
\tau_{\alpha}(x_1, \cdots ,x_d)=\sum_{i=1}^d\alpha_i\tau_k(x_i)=\sum_{i=1}^d \frac{l_i}{L} \Bigl(\frac{1}{k} \sum_{j=1}^kx_{jj}^i \Bigr)=\frac{1}{kL} \sum_{i=1}^d \sum_{j=1}^k l_ix_{jj}^i.
\end{equation}
Therefore $\phi$ is a unital completely positive $(\tau_\alpha,\tau_{kL})$-preserving *-homomorphism from  $(\underbrace{M_k(\mathbb{C}) \oplus \cdots \oplus M_k(\mathbb{C})}_d,\tau_\alpha)$ to $(M_{kL}(\mathbb{C}),\tau_{kL})$. By Lemma \ref{Preserv}, $T$ also has an exact factorization through $(M_n(\mathbb{C}) \otimes M_{kL}(\mathbb{C}),\tau_n\otimes \tau_{kL})$. By \cite[Corollary 3.5]{HM15}, $T \otimes S_{kL} \in \text{conv(Aut}(M_n(\mathbb{C}) \otimes M_{kL}(\mathbb{C})))$. Therefore $T\in\mathcal{J}_n\subset\mathcal{K}_n$.\vspace{2mm}\\
{\bf (Step 2)} Suppose that $\alpha:=(\alpha_1,\alpha_2)\in \mathbb{Q}_+^2$ with $\alpha_1+\alpha_2=1$ and assume that $T$ has an exact factorization through $(M_n(\mathbb{C}) \otimes (M_{k_1}(\mathbb{C}) \oplus M_{k_2}(\mathbb{C})),\tau_n\otimes \tau_\alpha)$ (In general, $k_1 \neq k_2$), where $\tau_\alpha$ is a normal faithful tracial state on $M_{k_1}(\mathbb{C})\otimes M_{k_2}(\mathbb{C})$ defined by
\begin{equation}
\tau_\alpha(x,y) :=\alpha_1 \tau_{k_1}(x)+ \alpha_2 \tau_{k_2}(y), \qquad (x,y)\in M_{k_1}(\mathbb{C})\otimes M_{k_2}(\mathbb{C}).
\end{equation} 
We define the following map:
\begin{equation}
\psi: M_{k_1}(\mathbb{C}) \oplus M_{k_2}(\mathbb{C}) \rightarrow M_{k_1k_2}(\mathbb{C}) \oplus M_{k_1k_2}(\mathbb{C}), \qquad
(x,y) \mapsto (\text{diag}(\underbrace{x, \cdots ,x}_{k_2}), \text{diag}(\underbrace{y, \cdots ,y}_{k_1})).
\end{equation}
It is clear that $\psi$ is a unital completely positive *-homomorphism. Consider two matrices $x=(x_{ij})_{1 \le i,j \le k_1} \in M_{k_1}(\mathbb{C})$ and $y=(y_{ij})_{1 \le i,j \le k_2} \in M_{k_2}(\mathbb{C})$. Then
\begin{equation}
\begin{split}
(\alpha_1 \tau_{k_1k_2} \oplus \alpha_2 \tau_{k_1k_2}) \circ \psi (x, y)&=(\alpha_1 \tau_{k_1k_2} \oplus \alpha_2 \tau_{k_1k_2}) (\text{diag}(\underbrace{x, \cdots ,x}_{k_2}), \text{diag}(\underbrace{y, \cdots ,y}_{k_1}))\\
&=\alpha_1 \tau_{k_1k_2}(\text{diag}(\underbrace{x, \cdots ,x}_{k_2}))+\alpha_2\tau_{k_1k_2}(\text{diag}(\underbrace{y, \cdots ,y}_{k_1}))\\
&=\alpha_1\Bigl(\frac{k_2}{k_1k_2}\sum_{i=1}^{k_1}x_{ii}\Bigr)+\alpha_2\Bigl(\frac{k_1}{k_1k_2}\sum_{j=1}^{k_2}y_{jj}\Bigr)\\
&=\frac{\alpha_1}{k_1}\sum_{i=1}^{k_1}x_{ii}+\frac{\alpha_2}{k_2}\sum_{j=1}^{k_2}y_{jj}\\
&=\alpha_1\tau_{k_1}(x)+\alpha_2 \tau_{k_2}(y)=\tau_{\alpha}(x,y),
\end{split}
\end{equation}
where $\alpha_1 \tau_{k_1k_2} \oplus \alpha_2 \tau_{k_1k_2}$ is a normal faithful tracial state on $M_{k_1k_2}(\mathbb{C}) \oplus M_{k_1k_2}(\mathbb{C})$ defined by
\begin{equation}
\alpha_1 \tau_{k_1k_2} \oplus \alpha_2 \tau_{k_1k_2}(x,y):=\alpha_1\tau_{k_1k_2}(x)+\alpha_2\tau_{k_1k_2}(y), \qquad (x,y)\in M_{k_1k_2}(\mathbb{C}) \oplus M_{k_1k_2}(\mathbb{C}).
\end{equation}
Therefore $\psi$ is a unital completely positive $(\tau_\alpha, \alpha_1\tau_{k_1k_2} \oplus \alpha_2\tau_{k_1k_2})$-preserving *-homomorphism. By using Lemma \ref{Preserv}, $T$ has an exact factorization through $(M_n(\mathbb{C}) \otimes (M_{k_1k_2}(\mathbb{C}) \oplus M_{k_1k_2}(\mathbb{C})),\tau_n\otimes(\alpha_1 \tau_{k_1k_2} \oplus \alpha_2 \tau_{k_1k_2}))$. Denote by
\begin{equation}
\alpha_1=\frac{l_1}{L},\hspace{1mm} \alpha_2=\frac{l_2}{L}, \hspace{2mm} l_1,l_2,L \in \mathbb{N} \text{ with } l_1+l_2=L.
\end{equation}
By using the first step, $T$ has an exact factorization through $(M_n(\mathbb{C}) \otimes M_{k_1k_2L}(\mathbb{C}),\tau_n\otimes \tau_{k_1k_2L})$, and therefore $T\in\mathcal{J}_n\subset\mathcal{K}_n$ by \cite[Corollary 3.5]{HM15}. \vspace{2mm}\\
{\bf (Step 3)} Assume that $\alpha:=(\alpha_1,\cdots ,\alpha_d) \in \mathbb{Q}_+^d$ with $\alpha_1+\cdots +\alpha_d=1$ and there exist some positive integers $k_1,\cdots ,k_d$ and a normal faithful tracial state $\tau_{\alpha}$ on $M_{k_1}(\mathbb{C})\oplus \cdots \oplus M_{k_d}(\mathbb{C})$ such that $T$ has an exact factorization through $(M_n(\mathbb{C}) \otimes (M_{k_1}(\mathbb{C}) \oplus \cdots \oplus M_{k_d}(\mathbb{C})),\tau_n\otimes\tau_{\alpha})$. By using repeatedly first step or second step, there exists positive integer $k$ such that $T$ also has an exact factorization through $(M_n(\mathbb{C}) \otimes M_{k}(\mathbb{C}),\tau_n\otimes \tau_k)$. Therefore $T\in\mathcal{J}_n\subset\mathcal{K}_n$ by \cite[Corollary 3.5]{HM15} again.
\end{proof}

% ------------------------------------------------------------------------

\subsection*{Acknowledgment}
This paper is a revised version of the master's thesis of the author. The author would like to express his heartly thanks to Professor Beno\^it Collins for carefully reading this paper and pointing out some inaccuracies. The author had a chance to visit Professor Magdalena Musat and Professor Mikael R\o rdam (in Copenhagen university) during November-December, 2016 with the support of the Top Global University project for Kyoto University. The author would appreciate their hospitality when the author was staying at Copenhagen University. In particular, thanks to their important advices, the author could advance researches related to the sections 4 and 5.

% 論文内容ここまで
%%%%%%%%%%%%%%%%%

\end{document}